\documentclass{amsart}
\usepackage{amssymb, mathrsfs, bbm}
\usepackage[T1]{fontenc}
\usepackage[sc, osf]{mathpazo}
\selectfont
\usepackage[protrusion=true, expansion=true]{microtype}

\usepackage[all, arc]{xy}
\SelectTips{eu}{}
\entrymodifiers={+!!<0pt,\fontdimen22\textfont2>}

\usepackage[pdftex, colorlinks=true, linkcolor=blue, citecolor=blue, linktocpage]{hyperref}

\makeatletter
\def\tagform@#1{\maketag@@@{\bfseries(\ignorespaces#1\unskip\@@italiccorr)}}
\renewcommand{\eqref}[1]{\textup{{\normalfont(\ref{#1}}\normalfont)}}
\makeatother

\theoremstyle{plain}
\newtheorem{thm}[equation]{Theorem}
\newtheorem{lemma}[equation]{Lemma}

\newtheorem{cor}[equation]{Corollary}

\theoremstyle{definition}

\theoremstyle{remark}

\theoremstyle{definition}

\def\11{\mathbf{1}}

\def\CC{\mathbf{C}}

\def\QQ{\mathbf{Q}}

\def\ZZ{\mathbf{Z}}

\def\gr{\mathrm{gr}}

\def\pt{\mathrm{Spec}(\CC)}

\newcommand{\mapright}[1]{\xrightarrow{#1}}

\title{Chow groups and equivariant geometry}
\author{Rahbar Virk}
\email{rsvirk@gmail.com}
\address{The Appalachians}
\begin{document}
\maketitle
%\setcounter{tocdepth}{1}
%\tableofcontents
\renewcommand{\thesubsection}{\textbf{\arabic{section}.\arabic{subsection}}}
%\renewcommand{\thesubsection}{\arabic{subsection}}
%{\let\thefootnote\relax\footnotetext{
%In a perfect world, a more appropriate title for this note would have been ``On the unreasonable effectiveness of the Eilenberg-Moore spectral sequence". Unfortunately, due to technical nuisances involved with using the Eilenberg-Moore spectral sequence in algebro-geometric contexts, it only makes an appearance in highly disguised form. See the proof of Lemma \ref{tatecor}, \S\ref{s:formality}, and the proof of Theorem \ref{algcycles}.
%}}
\subsection*{Introduction}
Throughout, `variety' will mean `separated scheme of finite type over $\mathrm{Spec}(\CC)$'. 
%Further, `$\pt$' will be used as a notational substitute for `$\Spec(\CC)$'.
%All algebraic groups making an appearance are tacitly linear algebraic. 
A $G$-variety will mean a variety endowed with the action of a linear algebraic group $G$.
The existence of functorial mixed Hodge structures on the rational cohomology, Borel-Moore homology, equivariant cohomology, etc., of a variety will be freely used (see \cite{D}). 
For homology (Borel-Moore, equivariant, etc.), $H_*$ will be called \emph{pure} if each $H_i$ is a pure Hodge structure of weight $-i$.

%Given a variety $X$, write $H^{BM}_*(X;\QQ)$ for its Borel-Moore homology with rational coefficients. Let $CH_i(X)$ denote the Chow group of cycles of dimension $i$, modulo rational equivalence. Set
%\[ CH_*(X)_{\QQ} = CH_*(X)\otimes \QQ.\]
%The following is the main result of this note.
\begin{thm}\label{algcycles}Let $X$ be a variety on which a linear algebraic group acts with finitely many orbits. If the (Borel-Moore) homology $H_*(X;\QQ)$ is pure (for instance, if $X$ is rationally smooth and complete), then the cycle class map
\[ CH_*(X)_{\QQ} \mapright{\sim} H_{*}(X; \QQ), \]
from rational Chow groups to homology, is a degree doubling isomorphism. 
\end{thm}
%For a smooth projective $G$-variety $X$, admitting finitely many orbits, the surjection $CH_i(X)_{\QQ} \twoheadrightarrow H_{2i}^{BM}(X; \QQ)$ is as predicted by the Hodge conjecture (see Lemma \ref{tatecor}).
This extends a result of Fulton-MacPherson-Sottile-Sturmfels \cite{FMSS} from solvable groups to arbitrary linear algebraic groups. The price paid is that while most of the results of \cite{FMSS} hold integrally, we deal exclusively with rational coefficients.
Our arguments are quite different from those in \cite{FMSS}. In particular, Theorem \ref{algcycles} is deduced from statements in the equivariant context.

For a $G$-variety $X$, write $A^G_*(X)_{\QQ}$ for its rational equivariant Chow groups.
%Set
%\[ A^G_*(X)_{\QQ} = A^G_*(X) \otimes \QQ.\]
Let $H_*^{G}(X;\QQ)$ denote the $G$-equivariant (Borel-Moore) homology of $X$, and let $W_{\bullet}$ be the weight filtration on $H_*^{G}(X;\QQ)$.
\begin{thm}\label{lowest}Let $G$ be a linear algebraic group acting on a variety $X$. Assume $X$ admits finitely many orbits. Then the cycle class map
\[ A^G_i(X)_{\QQ} \mapright{\sim} W_{-2i}H_{2i}^{G}(X; \QQ)\]
is an isomorphism for each $i\in \ZZ$.
\end{thm}
This is established by mimicking B. Totaro's arguments from \cite{To}. Combined with Lemma \ref{tatecor}, it yields the equivariant analogue of Theorem \ref{algcycles}.
\begin{cor}\label{maincor}Let $G$ be a linear algebraic group acting on a variety $X$. Assume $X$ admits finitely many orbits, and that $H^{G}_*(X;\QQ)$ is pure. Then the cycle class map
\[ A^G_*(X)_{\QQ} \mapright{\sim} H_{*}^{G}(X; \QQ) \]
is a degree doubling isomorphism.
\end{cor}
Now Theorem \ref{algcycles} follows via restriction from the equivariant to the non-equivariant context, using a result of M. Brion (Lemma \ref{chowfree}).

\subsection*{Acknowledgments: }
This note is the meandering offspring of a joint project with W. Soergel and M. Wendt. I am grateful to them for our continuing discussions.

\section*{Proofs}
\subsection*{Preliminaries}
%Even though the main results make no reference to them, the proofs utilize higher Chow groups/motivic homology (see Lemma \ref{induction}).
%
Let $X$ be a $G$-variety. Write $\bar H_{ G}^*(X; \QQ(j))$ for the equivariant motivic cohomology of $X$ (with `coefficients' in $\QQ(j))$. Write $\bar H^{G}_*(X; \QQ(j))$ for the equivariant motivic (Borel-Moore) homology of $X$.
%\footnote{
%Morally, $H^{\cM,G}_i(X; \QQ(j)) = \Hom_{DM_G(X)}(\const{X}[i](j), \DD\const{X})$, where $\DD$ denotes Verdier duality, $\const{X}$ is the unit object in $DM_G(X)$, and $DM_G(X)$ is an appropriate equivariant derived category of motivic sheaves (in the style of \cite{CD}). A rigorous construction of this category does not yet exist in the published literature (work in progress \cite{SVW}). Regardless, since here the interest is in $H^{\cM,G}_*(X)$, as opposed to the whole equivariant category, one can make do by using the non-equivariant version of the above formula (see \cite[Proposition 19.18]{MVW}) on the approximation spaces used to define equivariant Chow groups (see \cite{EG}).}
In terms of the higher equivariant Chow groups $A_p^G(X, k)$ of \cite{EG}:
\[ \bar H_i^{G}(X; \QQ(j)) = A_j^G(X, i-2j)\otimes \QQ.\]
In particular,
$\bar H^{ G}_{2i}(X; \QQ(i)) = A^G_i(X)_{\QQ}$.
%If $X$ is smooth and equidimensional, then \cite[\S2.7]{EG}:
%\[ H^{\cM, G}_i(X; \QQ(j)) = H^{2\dim(X)-i}_{\cM, G}(X; \QQ(\dim(X) -j)).\]
%
%For our purposes, the primary utility of motivic Borel-Moore homology/higher Chow groups is in them extending the basic exact sequence of Chow groups. Namely, if $Z\subset X$ is a $G$-stable closed subvariety, and $U=X-Z$ the open complement, then there is a long exact sequence, called the \emph{localization sequence}:%(see \cite[Lemma 4]{EG}):
%\[ \to H^{\cM, G}_{2i+1}(U;\QQ(i)) \to H^{\cM, G}_{2i}(Z;\QQ(i)) \to H^{\cM, G}_{2i}(X;\QQ(i)) \to H_{2i}^{\cM, G}(U;\QQ(i)) \to 0.\]
%
It will be notationally convenient to set
\[ A^i_G = \bar H^{2i}_{ G}(\pt; \QQ(i)).\]
%In other words, $A^*_G$ is the (rational) Chow ring of the classifying stack $BG$.
%
%\subsection{Restriction}
Given a group morphism $H\to G$, there is a restriction map $A^*_G \to A^*_H$.
%Taking $G^0\hookrightarrow G$ to be the inclusion of the identity component, one sees that $G/G^0$ acts on $A^*_{G^0}$, and the restriction map defines an isomorphism:\footnote{
%It is crucial that $\QQ$-coefficients are being used here.}
%\[ A^*_G \mapright{\sim} (A^*_{G^0})^{G/G^0}. \]
%This is a consequence of the corresponding property for the ordinary (i.e., non-equivariant) Chow ring \cite[Example 1.7.6]{Fu}.
%If $G$ is connected, and $U\subset G$ is the unipotent radical of $G$. Then $G/U$ is reductive, and restriction yields an isomorphism \cite[Lemma 2.18]{ToB}):
%\begin{equation}\label{arbconnected} A^*_{G/U} \mapright{\sim} A^*_G.\end{equation}
%This follows from $\AA^1$-homotopy invariance
%
%Combined, the above two facts allow us to reduce questions about arbitrary linear algebraic groups to connected reductive groups.
There are analogous restriction maps for motivic homology.
%Restricting to the identity gives the forgetful map
%$A_*^G(X)_{\QQ} \to CH_{*}(X)_{\QQ}$.
\begin{lemma}\label{chowfree}If $G$ is connected, then restriction induces an isomorphism:
\[ \QQ \otimes_{A^*_G} A^G_*(X)_{\QQ} \mapright{\sim} CH_*(X)_{\QQ}. \]
\end{lemma}
\begin{proof}
If $G$ is reductive, then this is \cite[Corollary 6.7(i)]{Btorus}. In general,
if $U\subset G$ is the unipotent radical, then $G/U$ is reductive, and restriction yields an isomorphism $A^*_{G/U} \mapright{\sim}A^*_G$. Similarly for motivic homology (see \cite[Lemma 2.18]{ToB}).
\end{proof}
%The `connected' assumption above is actually not necessary. This can be seen using transfers \cite[\S2.5]{ToB}. But we do not need this level of generality.
%Lemma \ref{chowfree} should be contrasted with the topological situation. Let
%\[ H^*_G = H^*_G(\pt;\QQ). \]
%Then restriction induces an isomorphism
%$\QQ \otimes_{H^*_G}^L H^{BM,G}_*(X;\QQ)\mapright{\sim} H_*^{BM}(X;\QQ)$.
%and `$\otimes^L$' needs to be interpreted in an appropriate dg-setting (see \cite[\S13]{BL}). 
%In general, the (dg)-derived `$\otimes^L$' cannot be replaced by `$\otimes$'.

There is a natural map
$\bar H^{ G}_i(X; \QQ(j)) \to W_{-2j}H^{G}_i(X; \QQ)$.
See \cite[\S4]{To} for a cogent explanation of this. 
%The resulting map $H^{\cM, G}_i(X; \QQ(j)) \to H^{BM,G}_i(X)$ is compatible with localization sequences, restriction maps, etc. 
The map 
\[ \bar H^{G}_{2i}(X; \QQ(i)) = A^G_i(X)_{\QQ} \to H^{G}_{2i}(X; \QQ) \]
is the \emph{cycle class map}. 
%It yields a degree doubling isomorphism %\cite[Proposition 6]{EG}, 
%\cite[Theorem 2.14]{ToB}:
%\begin{equation}\label{cycleiso} A^*_G \mapright{\sim} H^*_G, \end{equation}
%where $H^*_G$ denotes the $G$-equivariant cohomology ring of a point. That is,
%\[ H^*_G = H^*_G(\pt; \QQ).\]
\subsection*{Weak property}
A $G$-variety $X$ satisfies the \emph{weak property} if the cycle class map
\[ \bar H^{ G}_{2i}(X; \QQ(i)) = A_i^G(X)_{\QQ} \to W_{-2i}H^{G}_{2i}(X;\QQ) \]
is an isomorphism for each $i\in \ZZ$.

\subsection*{Strong property}
A $G$-variety $X$ satisfies the \emph{strong property} if it satisfies the weak property and the map
\[ \bar H^{ G}_{2i+1}(X; \QQ(i)) \to \gr_{-2i}^W H_{2i+1}^{G}(X;\QQ) \]
is surjective for each $i\in \ZZ$. Here $\gr^W_{\bullet}$ denotes the associated graded with respect to the weight filtration $W_{\bullet}$. 

\begin{lemma}\label{base}Let $G$ be a linear algebraic group, and let $K\subset G$ be a closed subgroup. Then $G/K$ satisfies the strong property (as a $G$-variety).
\end{lemma}

\begin{proof}
The map
$A^G_*(G/K)_{\QQ} \mapright{\sim} H_{*}^G(G/K;\QQ)$
is a degree doubling isomorphism (see \cite[Theorem 2.14]{ToB}).
\end{proof}

\begin{lemma}\label{induction}Let $X$ be a $G$-variety, $Z\subset X$ a $G$-stable closed subvariety, and $U = X - Z$ the open complement. If $U$ satisfies the strong property and $Z$ the weak, then $X$ satisfies the weak property.
\end{lemma}

\begin{proof}
We have a morphism of long exact sequences:
\[\xymatrixcolsep{.75pc}\xymatrixrowsep{1.0pc}\xymatrix{
\bar H^{ G}_{2i+1}(U; \QQ(i)) \ar[r]\ar[d]& \bar H^{G}_{2i}(Z;\QQ(i)) \ar[r]\ar[d]& \bar H^{ G}_{2i}(X; \QQ(i)) \ar[r]\ar[d]& \bar H^{ G}_{2i}(U; \QQ(i))\ar[d]\ar[r]&0 \\
\gr^W_{-2i}H^{G}_{2i+1}(U)\ar[r] & W_{-2i}H^{G}_{2i}(Z)\ar[r] & W_{-2i}H^{G}_{2i}(X) \ar[r]& W_{-2i}H^{G}_{2i}(U)\ar[r]&0
}\]
where `$\QQ$' has been omitted from the notation in the bottom row due to typesetting considerations.
%The top rightmost horizontal map is surjective because it is at the tail end of the localization sequence for Chow groups. The bottom rightmost horizontal map is surjective because $H^{BM,G}_{k}(Y;\QQ)$ has weights $\geq -k$, for all $k$ and $Y$.
%\footnote{
%Borel-Moore homology $H^{BM}_k(Y;\QQ)$ is dual to $H_c^k(Y;\QQ)$ (cohomology with compact support), and $H^k_c(Y;\QQ)$ has weights $\leq k$ \cite[\S8]{D0}.
%}
The first vertical map is surjective (strong property for $U$). The second and fourth vertical maps are isomorphisms by the weak property for $Z$ and $U$ respectively. So the third vertical map must also be an isomorphism.
\end{proof}

\subsection*{Proof of Theorem \ref{lowest}}
Combine Lemma \ref{base} and Lemma \ref{induction}.

\subsection*{Proof of Corollary \ref{maincor}}
Combine Theorem \ref{lowest} with the following observation.
\begin{lemma}\label{tatecor}
Let $X$ be a variety on which a linear algebraic group $G$ acts with finitely many orbits. Then the (Borel-Moore) homology $H_*(X;\QQ)$ is a sucessive extension of Hodge structures of type $(n,n)$.
\end{lemma}
\begin{proof}
We may assume $X=G/K$, where $K\subset G$ is a closed subgroup. Now $H^*(G/K;\QQ)$ is the $K$-equivariant cohomology of $G$. Consider the usual simplicial variety computing this (see \cite[\S6]{D}). Filtering by skeleta yields a spectral sequence whose $E_1$ entries are of the form $H^q(K^{\times p} \times G; \QQ)$ \cite[Proposition 8.3.5]{D}. Now 
%apply the K\"unneth formula and 
recall that the cohomology of a linear algebraic group is 
%a successive extension of Hodge structures of 
of type $(n,n)$ \cite[\S9.1]{D}.
\end{proof}

\subsection*{Proof of Theorem \ref{algcycles}}
We may assume that $G$ is connected. 
Let $H^*_G$ denote the equivariant cohomology ring of a point.
Purity and Lemma \ref{tatecor} imply that $H_*(X;\QQ)$ is concentrated in even degrees. Purity also implies that the natural map
\[ \QQ \otimes_{H^*_G} H_*^{G}(X;\QQ) \mapright{\sim} H_*(X;\QQ) \]
is an isomorphism. 
%The corresponding map on Chow groups
%\[ \QQ\otimes_{A^*_G} A_*^G(X)_{\QQ} \mapright{\sim} CH_*(X)_{\QQ} \]
%is also an isomorphism \eqref{chowfree}.
Further, the cycle class map
$A_*^G(X)_{\QQ} \mapright{\sim} H_{*}^{G}(X;\QQ)$
is an isomorphism by Corollary \ref{maincor}, since purity of $H_*(X;\QQ)$ implies purity of $H^{G}_*(X;\QQ)$.
Thus, combined with Lemma \ref{chowfree}, we obtain a commutative diagram:
\[\xymatrixcolsep{1pc}\xymatrixrowsep{1.5pc}\xymatrix{
A_*^G(X)_{\QQ} \ar[d]_{\sim}\ar[r] & \QQ\otimes_{A_G^*} A_*^G(X)_{\QQ} \ar[d]\ar[r]^-{\sim} & CH_*(X)_{\QQ} \ar[d] \\
H_{*}^{G}(X;\QQ) \ar[r] &  \QQ\otimes_{H_G^{*}} H_{*}^{G}(X;\QQ) \ar[r]^-{\sim}& H_{*}(X;\QQ) 
}\]
Consequently, $CH_*(X)_{\QQ}\mapright{\sim} H_{*}(X;\QQ)$ is an isomorphism.

\end{document}